\documentclass[12pt]{article}
\usepackage{fullpage}
\usepackage{mathrsfs}
\usepackage{amssymb}
\usepackage{amsmath}
\usepackage{graphicx}
\usepackage{mathtools}
\usepackage{amsmath,amsfonts}
\usepackage{todonotes}
\newtheorem{theorem}{Theorem}
\newtheorem{lemma}[theorem]{Lemma}
\newtheorem{definition}[theorem]{Definition}
\newtheorem{proposition}[theorem]{Proposition}
\newtheorem{corollary}[theorem]{Corollary}

\newenvironment{proof}{\noindent \noindent\relax{\sc
     Proof.}}{
     {\hfill$\Box$}\\
     }
\newcommand{\be}{\begin{equation}} \newcommand{\ee}{\end{equation}}
\newcommand{\ba}{\begin{align}} \newcommand{\ea}{\end{align}}
\newcommand{\baa}{\begin{align*}} \newcommand{\eaa}{\end{align*}}
\newcommand{\ben}{\begin{enumerate}} \newcommand{\een}{\end{enumerate}}
\newcommand{\bi}{\begin{itemize}} \newcommand{\ei}{\end{itemize}}

\newcounter{example}

\newcommand{\rd}{\mathrm{d}} 
\newcommand{\rP}{\mathrm{P}} 
\newcommand{\rE}{\mathrm{E}} 
\begin{document}

\title{
Perron-Frobenius theory for kernels and Crump-Mode-Jagers processes with macro-individuals
} 
\author{Serik Sagitov 
\\Chalmers University of Technology and University of Gothenburg
}
\maketitle
\begin{abstract}
Perron-Frobenius theory developed for irreducible non-negative kernels deals with so-called $R$-positive recurrent kernels. If kernel $M$ is $R$-positive recurrent, then the main result determines the limit of the scaled kernel iterations  $R^nM^n$ as $n\to\infty$. In the Nummelin's monograph \cite{N} this important result is proven using a regeneration method whose major focus is on $M$ having an atom. In the special case when $M=P$ is a stochastic kernel with an atom, the regeneration  method has an elegant explanation in terms of an associated split chain. 

In this paper we give a new probabilistic interpretation of the general regeneration method in terms of multi-type Galton-Watson processes producing clusters of particles. Treating clusters as macro-individuals, we arrive at a single-type Crump-Mode-Jagers process with a naturally embedded renewal structure. 
%
%
%
\end{abstract}
Keywords:  irreducible non-negative kernels, multi-type Galton-Watson process, $R$-positive recurrent kernel

\section{Introduction}
A Galton-Watson (GW) process describes random fluctuations of the numbers of independently reproducing particles counted generation-wise, see \cite{AN}. 
Given a measurable type space $(E,\mathcal E)$, the multi-type GW process is defined as a measure-valued Markov chain $\{\Xi_n\}_{n=0}^\infty $, where $\Xi_n(A)$ gives the number of $n$-th generation particles whose types lie in the set $A\in \mathcal E$, see  \cite[Ch 3]{Ha}. 
Given a current state
\[\Xi_n=\sum_{i=1}^{Z_n}\delta_{x_i},\quad \delta_x(A):=1_{\{x\in A\}},\]
where $Z_n=\Xi_n(E)$ is the number of particles in the $n$-th generation and $x_1,x_2,\ldots$ are the types of these particles, the next state of the Markov chain is determined in terms of the offspring to $Z_n$ particles
\[\Xi_{n+1}=\sum_{i=1}^{Z_n}\Xi^{(x_i)}_{i,n},\quad \Xi^{(x_i)}_{i,n}\stackrel{d}{=}\Xi^{(x_i)}.\]
The random measure $\Xi^{(x_i)}_{i,n}$, describing the allocation of a group of siblings over the type space $E$, is assumed to be independent of everything else except for the maternal type  $x_i$.

A key characteristic of the multi-type GW process is its {\it reproduction kernel} $M$ defining the expected number of offspring found in a given subset of the type space
\begin{equation}\label{rke}
M(x,A)=\rE[ \Xi^{(x)}(A)],\quad x\in E,\quad A\in \mathcal E,
\end{equation}
as a function of the maternal type $x$. 
Denote by  $M^n$ the iterations of the reproduction kernel:
\begin{equation}\label{Mndef}
M^0(x,A)=\delta_x(A),\quad  M^{n}(x,A)=\int M^{n-1}(y,A)M(x,\rd y),\quad n\ge1,
\end{equation}
here and elsewhere in this paper, the integrals are taken over the whole type space $E$, unless specified otherwise.
Then, for the multi-type GW process with the initial state $\Xi_0$, we get
$$\rE [\Xi_n(A)]=\int M^n(x,A)\mu_0(\rd x),\qquad \mu_0= \rE [\Xi_0].$$

The asymptotic properties of the multi-type GW  processes are studied on the basis of the Perron-Frobenius theorem dealing with the limiting behaviour of the expectation kernels and producing an asymptotic formula of the form
$$M^n(x,A)\sim \rho^n\dfrac{h(x)\pi(A)}{\int h(y)\pi(dy)},\quad n\to\infty,$$
see \cite[Ch 6]{Mode}. In the classical case of finitely many types, $M$ is a matrix and $\rho$ is its largest, the so-called Perron-Frobenius eigenvalue. Depending on whether $\rho<1$, $\rho=1$, or $\rho>1$, we distinguish among subcritical, critical, or supercritical GW processes.

The Perron-Frobenius theory for the irreducible non-negative kernels is build around the so-called  regeneration method, see  \cite{AN1} and especially \cite{N}. A key step of the regeneration method deals with $M$ having an atom, see Section \ref{SPF} for key definitions. 
%
In the special case, when  $M=P$ is a stochastic kernel with an atom, one can write
\begin{equation}\label{spl}
P(x,A)=p(x,A)+g(x)\gamma(A),
\end{equation}
where $\gamma(E)=1$, $0\le g(x)\le1$ and $p(x,E)=1-g(x)$ for all $x\in E$. The transition probabilities defined by such a kernel $P(x,dy)$ describe  a {\it split chain}, whose transition from a given state $x$ is governed either by $\gamma( \rd y)$ or $p(x, \rd y)/(1-g(x))$ depending on a random outcome of a $g(x)$-coin tossing  \cite[Ch. 4.4]{N}. After each $\gamma$-transition step, the future evolution of the split chain becomes independent from the past and present states, so that the sequence of such regeneration events forms a renewal process with a delay. Then, it remains to apply the basic renewal theory to establish the Perron-Frobenius theorem for stochastic kernels.

In this paper we suggest a probabilistic interpretation of the general regeneration method (when kernel $M$ is not necessarily stochastic) in terms of a certain class of multi-type GW processes which we call GW processes with clusters, see Section \ref{xip}.
In Section  \ref{eCrump-Mode-Jagers } we show that a GW process with clusters has an intrinsic structure of the single-type Crump-Mode-Jagers  (CMJ)  process with discrete time \cite{JS}. 
In Sections  \ref{sec} and \ref{genf} we give a proof of a suitable version of the  Perron-Frobenius theorem for the kernels with an atom, see Theorem \ref{PF}, using the regeneration property of the renewal process embedded into the CMJ  process. Section  \ref{xmpl} 
contains an illuminating example of a  GW process with clusters.


\section{ Irreducible kernels 
} \label{SPF}
In this section we give a summary of basic definitions and results 
presented in \cite{N}, including Theorems 2.1, 5.1, 5.2, 
and Propositions 2.4, 2.8, 3.4.

Consider a measurable type space $(E,\mathcal E)$ assuming that $\sigma$-algebra $\mathcal E$ is countably generated. We denote by $\mathcal M_+$ the set of  $\sigma$-finite measures $\phi$ on $(E,\mathcal E)$, and write $\phi\in\mathcal M^+$ if $\phi\in\mathcal M_+$ and $\phi(E)\in(0,\infty]$.

\begin{definition}
A (non-negative) kernel  on $(E,\mathcal E)$ is a map $M:E\times\mathcal E\to [0,\infty)$ such that for any fixed  $A\in \mathcal E$, the function $M(\cdot,A)$ is measurable, and on the other hand, $M(x,\cdot)\in\mathcal M_+$ for any fixed $x\in E$. For a pair $(x,A)\in(E,\mathcal E)$, we write $x\to A$ if
 \[M^n(x,A)>0 \text{ for some } n\ge1.\]
 Kernel $M$ is called irreducible, if there is such a measure $\phi\in\mathcal M^+$, that for any $x\in E$, we have $x\to A$ whenever $\phi(A)>0$. Measure $\phi$ is then called an irreducibility measure for $M$. 
\end{definition}
 If measure $\phi'\in\mathcal M^+$ is absolutely continuous with respect to an irreducibility measure $\phi$, then $\phi'$ is itself an irreducibility measure.
 For an irreducible kernel $M$, there always exists a {\it maximal irreducible measure}  $\psi$ such that  any other  irreducibility measure $\phi$ is absolutely continuous with respect to $\psi$. 

For an irreducible kernel $M$ with a maximal irreducible measure  $\psi$, there is a decomposition of the form
\begin{equation}\label{m0}
  M^{n_0} (x,A)= m(x,A) + g(x)\gamma(A),\quad \text{for all }x\in E, A\in \mathcal E,
 \end{equation}
where 
\begin{quote}
  $\gamma$ is an irreducibility measure for $M$, \\
 $g$ is a measurable non-negative function such that $\int g(x)\psi(dx)>0$, \\
$m$ is a another kernel on $(E,\mathcal E)$, \\
 $n_0$ is a positive integer number.
\end{quote}
\begin{definition}
 If \eqref{m0} holds with $n_0=1$, so that
\begin{equation}\label{mrg}
  M (x,A)= m(x,A) + g(x)\gamma(A),\quad x\in E,\quad A\in \mathcal E,
 \end{equation}
then the kernel $M$ is said to have an atom $(g,\gamma)$.
\end{definition}
Given \eqref{m0}, put
\begin{equation}\label{eqR}
 F(s)= \sum_{n=1}^\infty F_ns^{n},\quad
F_n= \iint g(y)M^{n-1}(x,\rd y)\gamma(d x).
\end{equation}

\begin{definition}\label{defR}
 Define the convergence parameter $R\in[0,\infty)$ of the irreducible kernel $M$ by
\[F(s)<\infty \text{ for } s<R, \text{ and }F(s)=\infty \text{ for } s>R.\]
If $F(R)<\infty$, then kernel $M$ is called $R$-transient, if $F(R)=\infty$, then kernel $M$ is called $R$-recurrent.
\end{definition}

\begin{definition}
 A non-negative measurable function $h$ which is not identically infinite is called $R$-subinvariant for $M$ if
 \[h(x)\ge R\int h(y)M(x,\rd y),\quad \text{for all } x\in E.\]
 An $R$-subinvariant function is called $R$-invariant if
 \[h(x)= R\int h(y)M(x,\rd y),\quad \text{for all } x\in E.\]
 
 A measure $\pi\in\mathcal M^+$ such that $\int g(y)\pi (dy)\in(0,\infty)$  is called $R$-subinvariant for $M$ if
 \[\pi(A)\ge R\int M(x,A)\pi(dx),\quad \text{for all } A\in \mathcal E.\]
 An $R$-subinvariant meaure is called $R$-invariant if
 \[\pi(A)= R\int M(x,A)\pi(dx),\quad \text{for all } A\in \mathcal E.\]
 \end{definition}
 
 Suppose $M$ is  $R$-recurrent. The function $h$ and the measure $\pi$ defined by 
\begin{align}
h(x)=\sum_{n=1}^\infty R^{nn_0}\int g(y)m^{n-1}(x,\rd y),\qquad \pi(A)=\sum_{n=1}^\infty R^{nn_0}\int m^{n-1}(x,A)\gamma(\rd x)   \label{Lus}
\end{align}
are $R$-invariant  for $M$, scaled in  such a way that
\begin{align}\label{earl}
\int h(x)\gamma(\rd x)= \int g(y)\pi(\rd y)=1.
\end{align}
For any $R$-subinvariant function $\tilde h$ satisfying $\int \tilde h(x)\gamma(\rd x)=1$, we have
\[\tilde h=h\quad \psi\text{-everywhere} \quad \text{and}\quad \tilde h\ge h\quad \text{everywhere}.\]
The measure $\pi$ is the unique $R$-subinvariant  measure satisfying \eqref{earl}.

\begin{definition}\label{posi}
An  $R$-recurrent kernel $M$ is called $R$-positive recurrent if the $R$-invariant function and measure $(h,\pi)$ satisfy $ \int h(y)\pi(dy)<\infty$. If $ \int h(y)\pi(dy)=\infty$, then $M$ is called $R$-null recurrent. 
\end{definition}

\begin{definition}
 Kernel $M$ has period $d$ if $d$ is the smallest positive integer such that there is a sequence of non-empty disjoint sets $(D_0,D_1,\ldots D_{d-1})$ having the following property
 \[ \text{if } x\in D_i, \text{ then } M(x,E\setminus D_j)=0 \quad \text{for }j=i+1\ ({\rm mod}\ d), \quad\ i=0,\ldots, d-1.\]
 We call kernel $M$ aperiodic if its period $d=1$.
\end{definition}
In the periodic case with $d\ge2$, provided $M$ is irreducible and satisfies \eqref{m0}, there is an index $i$, $0\le i\le d-1$, such that $g=0$ over all $D_j$ except $D_i$. Furthermore,
\[\gamma(E\setminus D_j)=0 \text{ for } j=i+n_0\ ({\rm mod}\ d).\]

\section{ GW processes with clusters
} \label{xip}
As will be explained later in this section, the following definition yields the above mentioned split chain construction in the particular case when $\rP(\Xi^{(x)}(E)=1)=1$.
\begin{definition}\label{split}
Consider a multi-type GW process $\{Z_n\}_{n=0}^\infty $ whose reproduction measure can be decomposed into a sum of a random number of integer-valued random measures
 \begin{equation}\label{deco}
\Xi^{(x)}=\xi^{(x)}+\sum_{i=1}^{N^{(x)}}\tau_i.
 \end{equation}
Let each $\tau_i$ be independent of everything else and have a common distribution $ \tau_i\stackrel{d}=\tau$.

(i) Such a multi-type GW process will be called a GW process with clusters.  

(ii) Each group of particles behind a measure $\tau_i$ in \eqref{deco} will be called a cluster, so that $N^{(x)}$ gives the number of clusters produced by a single particle of type $x$. Simple clusters correspond to the case $\rP(\tau(E)=1)=1$.

(iii) A multi-type GW process 
with the reproduction measure $\xi^{(x)}$ 
will be called a stem process.
\end{definition}
Given  \eqref{deco} and  
\begin{equation}\label{decc}
\rE \xi^{(x)}(A)=m(x,A),\quad \rE N^{(x)}=g(x),\quad \rE \tau(A)=\gamma(A),
 \end{equation}
by the total expectation formula, we see that the kernel  \eqref{rke} satisfies \eqref{mrg}. 
Note that we allow for dependence between $\xi^{(x)}$ and $N^{(x)}$.
Definition \ref{split} puts no restrictions on the reproduction kernel $m$ of the stem process. 
The example from Section \ref{xmpl} presents a case with $E=[0, \infty)$, where the kernel $m$ is reducible, in that for any ordered pair of types $(x,y)$, where $x<y$, type $x$ particles (within the stem process) may produce type $y$ particles but not otherwise.

Consider a GW process with simple clusters  such that 
\begin{align*}
& \rP(\xi^{(x)}(E)=0)=g(x),\quad \rP(\xi^{(x)}(E)=1)=1-g(x), \quad g(x)\in[0,1],\\
 &N^{(x)}=1_{\{\xi^{(x)}(E)=0\}},\quad x\in E,\\
 &\tau=\delta_{Y},\quad \rP(Y\in A)=\gamma(A).
\end{align*}
In this case each particle produces exactly one offspring, and the GW process tracks  the type of the regenerating particle. Using \eqref{mrg}, we find that $M=P$ is a stochastic kernel satisfying \eqref{spl} with 
$$p(x,A)=(1-g(x))\rP(\xi^{(x)}(A)=1|\xi^{(x)}(E)=1).$$
As a result we get a split chain corresponding to a stochastic kernel. Notice that the associated stem process is a pure death multi-type GW process.

An important family of GW processes with simple clusters is formed by  linear-fractional multi-type GW processes, see \cite{LS, Sa}. This family is framed by the following additional conditions 
\begin{align*}
& \rP(\xi^{(x)}(E)=0)+ \rP(\xi^{(x)}(E)=1)=1,\\
 &N^{(x)}=N\cdot 1_{\{\xi^{(x)}(E)=1\}},\qquad \text{ where $N$ has a geometric distribution},\\
 &\tau=\delta_{Y}.
\end{align*}
In this case \eqref{mrg} holds with
\[m(x,A)=\rP(\xi^{(x)}(A)=1),\quad g(x)=\rE N\cdot \rP(\xi^{(x)}(A)=1),\quad \gamma(A)=\rP(Y\in A).\]
Here again, the stem process is a pure death multi-type GW process. 

\section{Embedded CMJ process
} \label{eCrump-Mode-Jagers }

The key assumption of Definition \ref{split} guarantees that the procreation of particles constituting a cluster is independent of the other parts of the GW process with clusters. The main idea of this paper is to treat each cluster  as a newborn CMJ individual, which reminds the construction of macro-individuals in the sibling dependence setting of \cite{Ol}.  

Consider the stem process starting from a single cluster at time 0 and denote by  $L\in[1,\infty]$ its extinction time. Put $X_0=1$ and let 
$X_n$
stand for the number of new clusters generated at time $n$ by the particles in the stem process born at time $n-1$, $n\ge1$. Observe that
\begin{align*}
f_n&:= \rE(X_n)=\iint g(y)m^{n-1}(x,d y)\gamma(d x).
\end{align*}
We treat the random vector  $(X_1,\ldots, X_{L})$ as the life record of the initial individual in an embedded CMJ process, see Figure \ref{F1}. A CMJ individual during its life of length $L$ at different ages produces random numbers of offspring, cf \cite{JS}. Such independently reproducing CMJ individuals build a population with overlapping generations (in contrast to GW particles living one unit of time, so that there is no time  overlap between generations).

\begin{figure}
\centering
 \includegraphics[height=7cm]{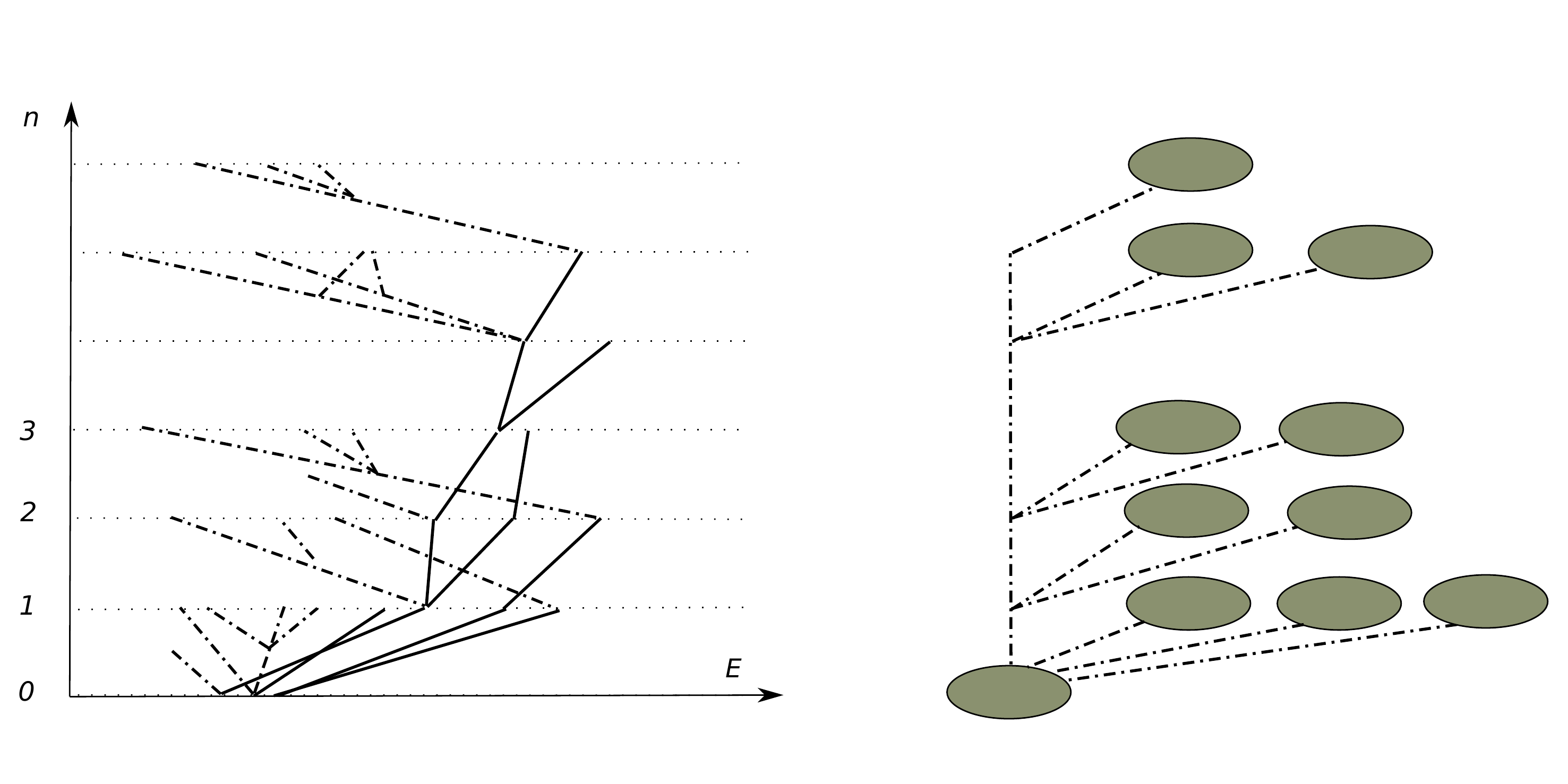}  
 \caption{Embedding a CMJ individual  into  a multi-type GW process stemming from a single cluster of size $Z_0=3$.  
 {\it Left panel}. 
 Solid lines represent the lineages of the stem process which dies out by time $L=6$.
Dashed lines delineate the daughter clusters directly generated by  the stem process. We see that $X_1=3$ with $\tau_1(E)=0$, $\tau_2(E)=1$, $\tau_3(E)=3$.
{\it Right panel}. The summary of the individual  life: 
$(X_1,\ldots,X_{L})=(3,2,2,0,2,1)$.
}
\label{F1}
\end{figure}

Throughout this paper we assume
\begin{equation}\label{rpo}
f(s_0)\in( 0 , \infty) \text{ for some }s_0>0,  \text{ where }f(s)= \sum_{n=1}^\infty f_ns^{n},
\end{equation}
so that on one hand, that $ f_n>0$ for some $n\ge1$, and on the other hand, 
the radius of convergence 
\[r= \inf \{ s \ge 0 \colon f(s) = \infty \}\]
 is positive.
The assumption $r>0$ prohibits very fast growing sequences of the type $f_n=e^{n^2}$.
\begin{definition}\label{dR}
Given  \eqref{rpo}, define a parameter $R\in(0,\infty)$ as  $R=r$ if  $f(r)<1$, and as the unique positive solution of the equation $f(R)=1$ if $f(r)\ge1$. 
\end{definition}
Since $f(R)\le1$, the sequence $(f_nR^{n})$ can be viewed as a (possibly defective) distribution on the lattice $\{1,2,\ldots\}$. This is the distribution of the inter-arrival time for the renewal process naturally embedded into the CMJ process defined above. 
The renewal process is interpreted as the consecutive ages at childbearing as one tracks a single ancestral lineage backwards in time.
 Given $f(R)=1$,  the mean inter-arrival time for the embedded renewal process equals
\[\sum_{n = 1}^\infty nf_nR^{n}=Rf'(R),\]
and is interpreted as the average age at childbearing or  the mean generation length for the CMJ process, see
\cite{J}.

Focussing on the current waiting time of such a discrete renewal process, we get an irreducible  Markov chain with the state space $\{0, 1,\ldots\}$. The following observation concerning this Markov chain  is straightforward. 
\begin{proposition}
The embedded renewal process is transient if $f(r)<1$, and  recurrent if  $f(r)\ge1$. Let $R$ be defined by Definition \ref{dR}. If $f(r)> 1$, then  $R\in(0,r)$, $f'(R)<\infty$, and the embedded renewal process is positive recurrent.
If $f(r)=1$, then the embedded renewal process is either positive recurrent or null recurrent depending on whether $f'(r)<\infty$ or $f'(r)=\infty$.
\end{proposition}
Let $W_n$ be the number of newborn individuals at time $n$ in the embedded CMJ process started from a single newborn individual, or in other words, the total number of clusters emerging at time $n$ in the original GW process starting from a single cluster. Clearly,
\begin{align*}
F_n&:= \rE(W_n)= \iint g(y)M^{n-1}(x,\rd y)\gamma(d x).
\end{align*}
\begin{theorem}\label{iv}
Consider a kernel $M$ with atom $(g,\gamma)$. Parameter $R$ from Definition \ref{dR} coincides with the convergence parameter of the kernel $M$.  Moreover,

(i) if $f(r)<1$, then $R=r$, $f(R)<1$, and $F(R)<\infty$, so that $M$ is $R$-transient,

(ii) if $f(r)\ge1$, then $f(R)=1$ and $F(R)=\infty$, so that $M$ is $R$-recurrent,

(iii) if $f(R)=1$, then either $f'(R)=\infty$ so that $M$ is $R$-null recurrent, or
$f'(R)\in(0,\infty)$, so that $M$ is $R$-positive recurrent.
\end{theorem}
\begin{proof}
Using the law of total expectation it is easy to justify the following recursion
$$F_n=f_n+f_{n-1}F_1+\ldots+f_1F_{n-1}.$$
This leads to the equality for generating functions 
$$F(s)=f(s)+f(s)F(s),$$
which yields
\begin{equation}\label{kFf}
 F(s)={f(s)\over 1-f(s)}\quad \text{for $s$ such that } f(s)<1.
\end{equation}
From here and in view of Definition \ref{defR}, it is obvious that the first statement is valid. Parts (i) and (ii) follow immediately. Part (iii) is proven in Section \ref{sec}.
\end{proof}

\noindent{\bf Remark.
} 
For a general starting configuration of particles $Z_0$, putting $\mu_0=\rE Z_0$, we get
\begin{align*}
\tilde f_n&:=\rE X_n=\iint g(y)m^{n-1}(x,d y)\mu_0(d x),\\
\tilde F_n&:= \rE Y_n=\iint g(y)M^{n-1}(x,\rd y)\mu_0(d x).
\end{align*}
The corresponding generating functions 
$$\tilde f(s)= \sum_{n=1}^\infty \tilde f_ns^{n},\quad \tilde F(s)= \sum_{n=1}^\infty \tilde F_ns^{n},$$ 
are connected by
\begin{equation}\label{tFf}
\tilde  F(s)={\tilde f(s)\over 1-f(s)}\quad \text{for $s$ such that } f(s)<1.
\end{equation}
(To obtain this relation, observe that
$$\tilde F_n=\tilde f_n+\tilde f_{n-1}F_1+\ldots+\tilde f_1F_{n-1},$$
which gives $\tilde F(s)=\tilde f(s)(1+F(s))$, and it remains to apply  \eqref{kFf}.)

 As mentioned above, under the special initial condition  $Z_0\stackrel{d}{=}\tau$, the embedded CMJ process starts from a single newborn individual. For a general $Z_0$, the embedded CMJ process has an immigration component characterised by the generating function $\tilde f(s)$. By immigration we mean  the inflow of new clusters generated by the stem process starting from $Z_0$ particles.
%
%
%
%

\section{Null and positive recurrence of a kernel with atom}\label{sec}

Consider a non-negative kernel $M$ with atom  $(g,\gamma)$, and put
\[M_s(x,A)=\sum_{n=1}^\infty s^nM^{n-1}(x,A),\quad m_s(x,A)=\sum_{n=1}^\infty s^nm^{n-1}(x,A),\quad s\ge0,\]
so that the earlier introduced generating functions $F$ and $f$ can be presented as
 \begin{align*}
F(s)=\iint g(y)M_s(x,\rd y)\gamma(\rd x),\quad f(s)=\iint g(y)m_s(x,\rd y)\gamma(\rd x).
\end{align*}
Denote
\begin{align}
h_s(x)=\int g(y)m_s(x,\rd y),\qquad \pi_s(A)=\int m_s(x,A)\gamma(\rd x),\label{Kus}
\end{align}
and observe that 
\begin{align*}
\int h_s(x)\gamma(\rd x)= \int g(y)\pi_s(\rd y)=f(s),\quad  \int h_s(y)\pi_s(\rd y)=s^2f'(s).
\end{align*}
The latter equality requires the following argument
\begin{align*}
 \int h_s(x)\pi_s(\rd x)&=\iiint   g(y)m_s(x,\rd y)m_s(z,\rd x)\gamma(\rd z)
 \\&=\iint   g(y)m_s^2(z,\rd y)\gamma(\rd z) =\sum_{n=1}^\infty ns^{n+1}f_{n}=s^2f'(s),
\end{align*}
where we used the relation
\begin{align*}
s^{-2}m_s^2(y,A)&= \int s^{-1}m_s(x,A)s^{-1}m_s(y,\rd x)=\sum_{n=0}^\infty\sum_{k=0}^\infty\int  s^n m^{n}(x,A)s^km^{k}(y,\rd x)\\
 &=\sum_{n=0}^\infty\sum_{k=0}^\infty s^{n+k}m^{n+k}(y,A)=\sum_{j=0}^\infty (j+1)s^{j}m^{j}(y,A).
\end{align*}

\begin{lemma}\label{eige}
Consider a kernel with atom \eqref{mrg}. If  a positive $s$ is such that $f(s)\le1$, then  the function $h_s$ and the measure $\pi_s$, defined by \eqref{Kus}, satisfy
 \begin{align}
\int h_s(y)M(x,\rd y)&=s^{-1}h_s(x)-(1-f(s))g(x),\label{st2}\\
\int M(y,A)\pi_s(\rd y)&=s^{-1}\pi_s(A)-(1-f(s))\gamma(A),\label{st1}
\end{align}
so that they are $s$-subinvariant function and measure for the kernel $M$.
\end{lemma}

\begin{proof}
By \eqref{mrg}, we have
 \begin{align*}
\int m_s(y,A)M(x,\rd y)&=\sum_{n=1}^\infty s^{n}m^{n}(x,A)+g(x)\int m_s(y,A)\gamma(\rd y)\\
&=s^{-1}m_s(x,A)-\delta_x(A)+g(x)\pi_s(A),
\end{align*}
which implies relation \eqref{st2}:
 \begin{align*}
\int h_s(y)M(x,\rd y)&=\iint g(w)m_s(y,\rd w)M(x,\rd y)=s^{-1}h_s(x)-g(x)+g(x)f(s).
\end{align*}
Similarly, from
 \begin{align*}
\int M(y,A)m_s(x,\rd y)&=\sum_{n=1}^\infty s^{n}m^{n}(x,A)+\gamma(A)\int g(y)m_s(x,\rd y) \\
&=s^{-1}m_s(x,A)-\delta_x(A)+\gamma(A)h_s(x),
\end{align*}
we  arrive at relation \eqref{st1}.
\end{proof}

Lemma \ref{eige} yields the following statement which in turn provides the proof of part (iii) of Theorem \ref{iv} (recall Definition \ref{posi}). 
\begin{corollary}
Consider an $R$-recurrent kernel $M$ with atom $(g,\gamma)$.
  If  $f(R)=1$, then $h=h_R$ and $\pi=\pi_R$ are $R$-invariant  function and measure satisfying relation
 \eqref{Lus} with $n_0=1$, relation \eqref{earl}, as well as
$$
\int h(y)\pi(\rd y)=R^2f'(R).
$$
%
\end{corollary}
Observe that 
\begin{equation}\label{hx}
 h(x)=\sum_{n=1}^\infty R^n\int g(y)m^{n-1}(x,d y)
\end{equation}
is the expected $R$-discounted number of clusters ever produced by the stem process starting from a single particle of type $x$.
From this angle, $h(x)$ can be interpreted as  the reproductive value of type $x$.
On the other hand, 
\begin{equation}\label{pia}
\pi(A)=\sum_{n=1}^\infty R^n\int m^{n-1}(x,A)\gamma(\rd x).
\end{equation}
 is the expected $R$-discounted number of particles whose type belongs to $A$ and which appear in the stem process starting from a single cluster of particles.
As shown next, see Theorem \ref{PF}, the measure $\pi$ can be viewed as an asymptotically stable distribution for  the types of particles in the GW process with clusters.

\section{Perron-Frobenius theorem for kernels with atom}\label{genf}
\begin{theorem}\label{PF}
Consider an aperiodic $R$-positive recurrent kernel $M$ with atom $(g,\gamma)$. Let $h$ and $\pi$ be given by \eqref{hx} and \eqref{pia}. If $(x,A)$ are such that
\begin{equation}\label{extra}
R^{n}m^n(x,A)\to0,\quad n\to\infty,
\end{equation}
 then 
\begin{equation}\label{pefr}
 R^{n}M^n(x,A)\to \dfrac {h(x)\pi(A)}{R^2f'(R)} ,\quad n\to\infty.
\end{equation}
If $h(x)<\infty$, then condition \eqref{extra} holds for any $A$ such that 
\begin{equation}\label{eps}
A\subset\{y: g(y)\ge\epsilon\} \text{ for some }\epsilon>0.
\end{equation}
\end{theorem}
To prove this result we need two lemmas. In the end of this section we give a remark addressing condition \eqref{extra}.

\begin{lemma}
Consider a kernel $M$ with atom $(g,\gamma)$.  If  $s>0$ is such that $f(s)<1$, then  
 \begin{align}
M_s(x,A)&=m_s(x,A)+{h_s(x)\pi_s(A)\over 1-f(s)}\quad \text{for all }x\in E, A\in\mathcal E. \label{Ks}
\end{align}
\end{lemma}
\begin{proof}
 By \eqref{mrg}, we have the recursion
%
\begin{align*}
M^n(x,A)&=g(x)\int M^{n-1}(y,A)\gamma(\rd y)+\int M^{n-1}(y,A)m(x,\rd y)
\\
&=g(x)\int M^{n-1}(y,A)\gamma(\rd y)+\int g(y)m(x,\rd y)\int M^{n-2}(z,A)\gamma(\rd z)\\
&\quad +\int M^{n-2}(z,A)m^2(x,dz)\\
 &=\sum_{i=1}^{n}\int g(y)m^{i-1}(x,\rd y)\int M^{n-i}(y,A)\gamma(\rd y)+m^{n}(x,A),
\end{align*}
which 
in terms of generating functions  gives
 \begin{align*}
M_s(x,A)&=m_s(x,A)+h_s(x)\int M_s(y,A)\gamma(\rd y),
\end{align*}
and after integration,
 \begin{align*}
\int M_s(x,A)\gamma(\rd x)&={\pi_s(A)\over 1-f(s)}.
\end{align*}
Combining the last two relations we get
\eqref{Ks}. Observe also that the last formula yields \eqref{kFf}.

\end{proof}

\begin{lemma} \label{Fe} 
Let 
$$a(s)=\sum_{n=0}^\infty a_ns^n,\quad b(s)=\sum_{n=0}^\infty b_ns^n,\quad c(s)=\sum_{n=0}^\infty c_ns^n,$$ 
be three generating functions for non-negative sequences connected by
$$c(s)={b(s)\over1-a(s)}.$$
If  sequence $\{a_n\}$ is aperiodic with $a(1)=1$, $a'(1)\in(0,\infty)$, then 
 \[c_n\to 
   \dfrac {b(1)}{a'(1)},\quad n\to\infty.
\]
\end{lemma}
\begin{proof}
 This is a well-known result from Chapter XIII.4 in \cite{Fe1}.
\end{proof}

\

\noindent{\sc Proof of Theorem \ref{PF}}. $R$-positive recurrence implies  $f(R)=1$ and $f'(R)\in(0,\infty)$.
Due to $f(R)=1$, we can rewrite \eqref{Ks} as 
\[M_{\hat s}(x,A)-m_{\hat s}(x,A)={b(s)\over1-a(s)},\]
 where $\hat s=sR$ and
\[a(s)=f(sR),\quad b(s)=h_{\hat s}(x)\pi_{\hat s}(A),\]
so that $a'(1)=Rf'(R)$, $b(1)=h(x)\pi(A)$. Applying Lemma \ref{Fe}, we find that as $n\to\infty$,
 \[R^{n}(M^{n}(x,A)-m^{n}(x,A))\to 
\dfrac {h(x)\pi(A)}{R^2f'(R)}.
\]
which combined with condition \eqref{extra} yields the main assertion. The stated sufficient condition for  \eqref{extra} is verified using
\[\sum_{n=1}^\infty R^nm^{n-1}(x,A)\le \sum_{n=1}^\infty R^n\int1_{\{y:g(y)>\epsilon\}}m^{n-1}(x,dy)\le\epsilon^{-1}h(x)<\infty.
\]
\hfill$\Box$

\noindent{\bf Remark}. To illustrate the role of the condition \eqref{extra}, consider the kernel \eqref{mrg}  with
$$m(x,A)=g_1(x)\gamma_1(A),$$
assuming 
$$\int g_1(x)\gamma_1(dx)=a_1,\quad \int g(x)\gamma(dx)=a,\quad \int g_1(x)\gamma(dx)= \int g(x)\gamma_1(dx)=0,$$
where $a_1>a>0$.
In this particular case, we have
$$M^n(x,A)=m^n(x,A)+a^ng(x)\gamma(A),\quad m^n(x,A)=a_1^ng_1(x)\gamma_1(A),$$
and clearly,
\[
M^n(x,A)\sim\left\{
\begin{array}{ll}
 a_1^ng_1(x)\gamma_1(A),  &  \text{if }g_1(x)\gamma_1(A)>0,   \\
 a^ng(x)\gamma(A), &   \text{if }g_1(x)\gamma_1(A)=0 \text{ and }g(x)\gamma(A)>0.
 \end{array}
\right.
\]
Turning to the generating function defined by \eqref{eqR} we find
 \[
F_n= \iint g(y)M^{n-1}(x,\rd y)\gamma(d x)=a^{n+1},\quad F(s)={a^2s\over 1-as}.
\]
This yields $R=a^{-1}$ and we see that condition \eqref{extra} is not valid for $(x,A)$ such that $g_1(x)\gamma_1(A)>0$. On the other hand, if $g(x)<\infty$ and $A$ satisfies \eqref{eps}, then 
\[0=\int g(x)\gamma_1(dx)\ge \int_A g(x)\gamma_1(dx)\ge \epsilon \gamma_1(A), \]
so that $\gamma_1(A)=0$ and therefore $R^nM^n(x,A)\to g(x)\gamma(A)$.

\section{3-parameter GW process with clusters
} \label{xmpl}

Here we construct a transparent example  of a GW process  with clusters having the type space $E = [0, \infty)$. Its positive recurrent reproduction kernel 
is fully specified by just  three parameters $a,c\in(0,\infty)$, and $b\in(-1,\infty)$:
\[M(x,\rd y)=ae^{x-y}1_{\{y\ge x\}}\rd y+ce^{-bx}\delta_0(\rd y).\]
This kernel satisfies \eqref{mrg} with 
\begin{equation}\label{emgg}
 m(x,\rd y)=ae^{x-y}1_{\{y\ge x\}}\rd y,\quad g(x)=ce^{-bx},\quad \gamma(A) =  \delta_0(A),
\end{equation}
implying that each cluster consists of a single particle of type 0. 

The full specification of our example refers to a continuous time  Markov branching process modeling the size of a population of {\it Markov particles}
having the unit life-length mean and offspring mean $a$. 
The main idea is to count the Markov particles generation-wise, and to define the type of a Galton-Watson particle as the birth-time of the corresponding  Markov particle. The corresponding stem process $\{\xi_n\}_{n\ge0}$ is defined by 
\begin{quote}
$\xi_n(A)=$ the number of $n$-generation Markov particles born in the time period $A$, 
\end{quote}
so that its conditionally on the parent's birth time $x$, 
 \[m^n(x,[0,t])=  a^n \mathbb P(x+T_1+\ldots+T_n\le t)=  a^n \mathbb P(N_{t-x}\ge n), \quad \text{for }t>x,
  \]
  where  $T_i$ are independent exponentials with unit mean and $\{N_t\}_{t\ge0}$ is the standard Poisson process. 

\begin{proposition}
Consider the above described multi-type GW process with clusters characterised by \eqref{emgg}. Then we have
 \begin{align}\label{tre}
f(s)
={rcs\over r- s}, \qquad
r={1+b\over a}, \qquad R={r\over 1+cr}.
\end{align}
The process is supercritical if  $c>{r-1\over r}$, critical if $c={r-1\over r}$, and subcritical if $c<{r-1\over r}$.

Convergence \eqref{pefr} holds  for $A=[0,t]$, $t\in [0,\infty)$, with the right hand side equal to
 \[e^{-bx}(R\delta_0(\rd y)+aR^2e^{(aR-1)y}\rd y).\]
If  $Ra<1$, then \eqref{pefr} holds even for $A=E$ with  the right hand side equal to ${Re^{-bx}\over 1-aR}$.

\end{proposition}
\begin{proof}
Referring to the underlying Poisson process, we find that for $s\ne1/a$,
\begin{align*}
m_s(0,[0,t])&= s\sum _{n=0}^\infty s^n a^{n}\sum _{k=n}^\infty \mathbb P(N_{t}=k)= s \sum _{k=0}^\infty  \mathbb P(N_{t}=k){ 1-(as)^{k+1} \over 1-as}\\
&  = { s  \over 1-as}(1- as\mathbb E(as)^{N_t})
= { s  \over 1-as}(1- ase^{t(as-1)}).
\end{align*}
More generally, we have
$$
m_s(x,[0,t])= m_s(0,[0,t-x])= { s \over 1-as}(1- ase^{(t-x)(as-1)})1_{\{t\ge x\}} ,
$$
so that
\begin{align*}
m_s(x,\rd y)=s\delta_x(\rd y)+as^2e^{(as-1)(y-x)}1_{\{y\ge x\}}\rd y.
\end{align*}

By \eqref{Kus}
\begin{align*}
h_s(x)=\int g(y)m_s(x,\rd y)= sce^{-bx}+ cas^2\int_{ 0}^\infty  e^{h(as-1)}e^{-b(x+u)}du
=f(s)e^{-bx},
\end{align*}
where $f(s)$ satisfies \eqref{tre}. Since $f(r)=\infty$, the stated value $R={r\over 1+cr}$ is found from the equation $f(R)=1$.

Applying once again \eqref{Kus}, we find 
\begin{align*}
\pi_s(\rd y)=\int m_s(x,\rd y)\gamma(\rd x)=m_s(0,\rd y)=s\delta_0(\rd y)+as^2e^{(as-1)y}\rd y.
\end{align*}
To check this and previously obtained expressions, we verify the general  formula for the integral
\begin{align*}
 \int h_s(x)\pi_s(\rd x)&=sf(s)+ f(s)as^2\int e^{-(1+b)x}e^{asx}\rd x={rsf(s)\over r-s}={r^2s^2\over (r-s)^2}=s^2f'(s).
 \end{align*}
 
With
\begin{align*}
h(x)&= e^{-bx},\qquad \pi(\rd x)=R\delta_0(\rd x)+aR^2e^{(aR-1)x}\rd x,
\end{align*}
Theorem \ref{PF} specialised to the current example says that for $t\in [0,\infty)$,
\begin{align*}
R^{n}M^n(x,[0,t])&\to e^{-bx}(R+aR^2\int_0^t e^{(aR-1)y}\rd y)\\
&=e^{-bx}(R+ {aR^2\over aR-1}(e^{(aR-1)t}-1))=e^{-bx}{aR^2e^{(aR-1)t}-R\over aR-1},\quad n\to\infty.
 \end{align*}
If $aR<1$ and $A=E$, then  condition \eqref{extra} holds since
$$\pi(E)=R+{aR^2\over 1-aR}={R\over 1-aR}<\infty,$$ 
and
$R^{n}m^n(x,E)=(Ra)^n\to0$.
\end{proof}

\noindent {\bf Remark.} If we further specialize this example by letting the stem process to be the Yule process, then we have $a=2$. If furthermore, $b=2$ and $c<{r-1\over r}={1\over 3}$, then the corresponding GW process with clusters is subcritical, despite the total number of particles in the Yule process is infinite.\\

\noindent {\bf Acknowledgements.} The author is grateful for critical remarks of an anonymous reviewer of an earlier version of the paper.


\begin{thebibliography}{00}

\bibitem{AN} {\sc Athreya, K. and Ney, P.} (1972) {\em Branching processes},  John Wiley \& Sons, London-New York-Sydney.
 \bibitem{AN1} {\sc Athreya, K. and Ney, P.} (1982) A Renewal Approach to the Perron-Frobenius Theory of Non-negative Kernels on General State Spaces.
{\em Mathematische Zeitschrift} 179,  507--530.

\bibitem{Fe1}{\sc Feller, W.} (1959). {\em An introduction to probability theory and its applications}, Vol I, 2nd ed. John Wiley \& Sons, London-New York-Sydney.
\bibitem{J} {\sc Jagers, P. } (1975) {\em Branching processes with biological applications}, Wiley, New-York. 


\bibitem{JS} {\sc Jagers, P. and Sagitov, S.} (2008) General branching processes in discrete time as random trees. {\em Bernoulli} 14, 949--962. 
%
%

%

\bibitem{Ha} {\sc Harris, T. E.} (1963) {\em The Theory of Branching Processes}, Springer, Berlin.

\bibitem{LS} {\sc Lindo, A. and Sagitov, S.} (2018)  General linear-fractional branching processes with discrete time.  {\em Stochastics} {\bf 90},  364--378.




\bibitem{Mode}{\sc Mode, C.J.} (1971) Multitype branching processes: theory and applications.
Volym 34 av Modern analytic and computational methods in science and mathematics. American Elsevier Pub. Co., 

\bibitem{Ol} {\sc Olofsson, P.} (1996) Branching processes with local dependencies. {\em Ann. Appl. Probab. } {\bf 6},  238--268. 

\bibitem{N} {\sc Nummelin, E.} (1984) {\em General Irreducible Markov Chains and Non-negative Operators}, Cambridge University Press, London.


\bibitem{Sa} {\sc Sagitov, S.} (2013) Linear-fractional branching processes with countably many types. {\em Stoch. Proc. Appl.} 123, 2940--2956. 




\end{thebibliography}
\end{document}